\documentclass{amsart}
\usepackage[utf8]{inputenc}

\usepackage{fullpage}
\usepackage{tabularx}
\usepackage{arydshln, tikz, amsfonts}
\usepackage{graphicx}
\usepackage[mathcal]{euscript}
\usepackage{float}
\usepackage{color}
\usepackage{tikz}
\usepackage{verbatim}
\usepackage{enumerate}
\usepackage{amsmath}
\usepackage{subfiles}
\usepackage{multirow}
\usepackage[hidelinks]{hyperref}
\usepackage{stackengine}
\stackMath

\newtheorem{theorem}{Theorem}[section]
\newtheorem*{theorem-non}{Theorem}
\newtheorem{lemma}[theorem]{Lemma}
\newtheorem{corollary}[theorem]{Corollary}
\newtheorem{proposition}[theorem]{Proposition}

\newtheorem{conjecture}[theorem]{Conjecture}

\theoremstyle{definition}
\newtheorem{definition}[theorem]{Definition}
\newtheorem{example}[theorem]{Example}

\theoremstyle{remark}

\numberwithin{equation}{section}

\newcommand{\Z}{\mathbb{Z}}

\newcommand{\braid}{\operatorname{Braid}}
\newcommand{\tbi}{\operatorname{tbi}}

\allowdisplaybreaks

\newcommand{\var}{\operatorname{Var}}
\newcommand{\sign}{\operatorname{sign}}

\definecolor{blue}{rgb}{0,0,1}

\title{The distribution of braid indices of 2-bridge knots}
\author{Tobias Clark, Jeremy Frank, Adam M. Lowrance}

\begin{document}

\begin{abstract}
In this article we study the braid indices of 2-bridge knots with a fixed crossing number $c$. We show that the average braid index of the set of $2$-bridge knots of crossing number $c$ is asymptotically linear, approaching $\frac{c}{3}+\frac{11}{9}$. Additionally, we show that the variance of the braid indices of the set of $2$-bridge knots of crossing number $c$ is also asymptotically linear, approaching $\frac{2c}{27} - \frac{10}{81}$. Finally, we find a formula for the number $k_{c,b}$ of $2$-bridge knots with crossing number $c$ and braid index $b$, and show that for any fixed $c$, the braid index where $k_{c,b}$ achieves its maximum is $b=\left\lceil \frac{c}{3}\right\rceil +1$. 
\end{abstract}

\maketitle

\section{Introduction}
\label{sec:intro}

A \textit{2-bridge knot} is a nontrivial knot with a representative such that projection to one coordinate yields only two maxima and two minima as critical points. Every $2$-bridge knot is the closure of a rational tangle, and Schubert \cite{Schubert} classified $2$-bridge knots in terms of their rational tangle representations. Recent work \cite{BKLMR, CohenBound, SuzTran, RayDiao, CoLow, CoLow2} has explored properties of the probability distribution of the genera of $2$-bridge knots with fixed crossing number $c$. In this article, we apply techniques from these sources to study the probability distribution of braid indices of $2$-bridge knots.

Alexander \cite{Alexander} proved that every knot can be represented as the closure of a braid. The \textit{braid index} of a knot is the minimum number of strands of any braid whose closure is the knot. Murasugi \cite{Murasugi} computed the braid index of a $2$-bridge knot; see Theorem \ref{thm:braid} for a precise statement.

Define $\mathcal{K}_c$ to be the set of $2$-bridge knots of crossing number $c$ where only one of a knot and its mirror image is in $\mathcal{K}_c$. For example, $|\mathcal{K}_3|=1$ because only one of the left-handed or right-handed trefoil is in $\mathcal{K}_3$. Define $\mathcal{K}_{c,b}$ to be the subset of $\mathcal{K}_c$ consisting of knots of braid index $b$, and let $k_{c,b}=|\mathcal{K}_{c,b}|$. Our first result gives a formula for the number $k_{c,b}$ of $2$-bridge knots with crossing number $c$ and braid index $b$.
\begin{theorem}
    \label{thm:number}
    The number $k_{c,b}$ of $2$-bridge knots with crossing number $c$ and braid index $b$ is
    \[k_{c,b} = \begin{cases}
        1& \text{if $3\leq c$, $c$ is odd, and $b=2$,}\\
        2^{b-4}\binom{c-b}{b-2} & \text{if $3\leq c$, $3\leq b \leq \left\lceil\frac{c+1}{2}\right\rceil$, and $c+b$ is even,}\\
        2^{b-4}\binom{c-b}{b-2} + 2^{\frac{b-5}{2}}\binom{\frac{c-b-1}{2}}{\frac{b-3}{2}}& \text{if $3\leq c$, $3\leq b \leq \left\lceil\frac{c+1}{2}\right\rceil$, $c$ is even, and $b$ is odd,}\\
        2^{b-4}\binom{c-b}{b-2} + 2^{\frac{b-4}{2}}\binom{\frac{c-b-1}{2}}{\frac{b-2}{2}}& \text{if $3\leq c$, $3\leq b \leq \left\lceil\frac{c+1}{2}\right\rceil$, $c$ is odd, and $b$ is even,}\\
        0& \text{otherwise.}
        \end{cases}\]
\end{theorem}

The \textit{mode} of a finite sequence $(a_1,\dots,a_n)$ of integers is an index $m$ where $a_m\geq a_i$ for all 
$i$ with $1\leq i \leq n$. In our next result, we find the mode of the sequence $(k_{c,b})_{b=2}^n$ where $n=\left\lceil \frac{c+1}{2}\right\rceil$ for each fixed crossing number.
\begin{theorem}
\label{thm:mode}
Let $c\geq 3$, and let $n=\left\lceil\frac{c+1}{2}\right\rceil$. The mode of the sequence $(k_{c,b})_{b=2}^n=(k_{c,2},k_{c,3},\dots,k_{c,n})$ of the number of 2-bridge knots with crossing number $c$ and braid index $b$ is $b=\left\lceil \frac{c}{3}\right\rceil +1.$
\end{theorem}
In Conjecture \ref{conj:median}, we conjecture that the median braid index is also $b=\left\lceil \frac{c}{3}\right\rceil +1.$

We also find the average and the variance of the braid indices of $2$-bridge knots with a fixed crossing number $c$. Suzuki and Tran \cite{SuzTran2} independently computed the average braid index. Define the discrete random variable $\braid_c:\mathcal{K}_c\to\mathbb{Z}$ by setting $\braid_c(K)$ be the braid index of the $2$-bridge knot $K$. We call $\braid_c$ the \textit{braid index of 2-bridge knots with crossing number $c$ random variable}, or just the \textit{braid index random variable}, for short. The average value of the braid indices of all $2$-bridge knots with crossing number $c$ is the expected value $E(\braid_c)$ of the braid index random variable. 
\begin{theorem}
\label{thm:average}
    The average braid index $E(\braid_c)$ of $2$-bridge knots with crossing number $c$ is
    \[E(\braid_c) = \frac{c}{3}+\frac{11}{9} + \begin{cases}
        \frac{2^{\frac{c}{2}}+9c-16}{9\left(2^{c-2}+2^{\frac{c-2}{2}}\right)}& \text{if $c\equiv 0$ mod $4$,}\\
       \frac{19-9c-2^{\frac{c+3}{2}}}{9\left(2^{c-2}+2^{\frac{c-1}{2}}\right)}&\text{if $c\equiv 1$ mod $4$,}\\
       \frac{2^{\frac{c}{2}}+3c-8}{9\left(2^{c-2}+2^{\frac{c-2}{2}}-2\right)} & \text{if $c\equiv 2$ mod $4$,}\\
       \frac{5-3c-2^{\frac{c+3}{2}}}{9\left( 2^{c-2} + 2^{\frac{c-1}{2}}+2\right)}& \text{if $c\equiv 3$ mod $4$.}
       \end{cases}\]
Hence the average braid index $E(\braid_c)$ of $2$-bridge knots with crossing number $c$ approaches $\frac{c}{3}+\frac{11}{9}$ as $c\to\infty$.
\end{theorem}

The variance of the braid indices of all $2$-bridge knots with crossing number $c$ is the variance $\var(\braid_c)$  of the braid index random variable.
\begin{theorem}
    \label{thm:variance}
    The variance $\var(\braid_c)$ of the braid index random variable is
    \[\var(\braid_c) = \frac{2c}{27}- \frac{10}{81} + \varepsilon(c),\]
    where
    \[\varepsilon(c) = 
    \begin{cases}
        \frac{(3c-13)2^{\frac{3c}{2}}+(21c-74)2^c-(42c-40)2^{\frac{c}{2}}-(324c^2-1152c+1024)}{81\left(2^{2c-2}+2^{\frac{3c}{2}}+2^c\right)}&\text{if $c\equiv 0$ mod $4$,}\\
        \frac{(3c-1)2^{\frac{3c+1}{2}}+(15c-13)2^c-(69c-175)2^{\frac{c+3}{2}}-(324c^2-1368c+1444)}{81\left(2^{2c-2}+2^{\frac{3c+1}{2}}+2^{c+1}\right)}&\text{if $c\equiv 1$ mod $4$,}\\
        \frac{(3c-13)2^{\frac{3c}{2}}-(18c^2-105c+142)2^c-(18c^2-75c+28)2^{\frac{c+2}{2}}+(108c^2-600c+640)}{81\left(2^{2c-2}+2^{\frac{3c}{2}}-3\cdot 2^c-2^{\frac{c+6}{2}}+16\right)}
&\text{if $c\equiv 2$ mod $4$,}\\
        \frac{ (3c-1)2^{\frac{3c+1}{2}}+(18c^2-105c+97)2^c +(18c^2-129c+169)2^{\frac{c+3}{2}}+(108c^2-816c+964)}{81\left(2^{2c-2}+2^{\frac{3c+1}{2}}+6\cdot 2^c+4\cdot 2^{\frac{c+3}{2}}+16\right)}&\text{if $c\equiv 3$ mod $4$.}
    \end{cases}\]
    Hence the variance $\var(\braid_c)$ approaches $\frac{2c}{27}-\frac{10}{81}$ as $c\to\infty$.
\end{theorem}

The proofs of Theorems \ref{thm:number} through \ref{thm:variance} rely on recursive formulas for quantities related to even continued fraction representations of 2-bridge knots (see Definition \ref{def:even}). The closed formulas in Theorems \ref{thm:number}, \ref{thm:average}, \ref{thm:variance} look somewhat unwieldy, but can be verified in a straightforward fashion using these recursive formulas derived in Sections \ref{sec:number}, \ref{sec:average}, and \ref{sec:variance}.

This article is organized as follows. In Section \ref{sec:background}, we recall the continued fraction representation of 2-bridge knots and how the crossing number and braid index of a knot can be recovered from certain continued fraction representations. In Section \ref{sec:number}, we find the number of 2-bridge knots with crossing number $c$ and braid index $b$, proving Theorem \ref{thm:number}. In Section \ref{sec:mode}, we find the mode of the braid indices of $2$-bridge knots, proving Theorem \ref{thm:mode}. In Section \ref{sec:average}, we compute the average braid index of $2$-bridge knots with crossing number $c$, and in Section \ref{sec:variance}, we compute the variance of the braid indices of $2$-bridge knots with crossing number $c$.\medskip

\textbf{Acknowledgments.} The authors thank J\'ozef Przytycki for suggesting this problem. This paper is the result of a summer research project in Vassar College's Undergraduate Research Science Institute.

\section{Background}
\label{sec:background}

In this section, we recall some facts about 2-bridge knots, their crossing numbers, and their braid indices. Our conventions largely follow Cromwell \cite{Cromwell}. Every 2-bridge knot corresponds to a rational number $\frac{p}{q}$ that can be expressed as a continued fraction
\[\frac{p}{q} = 
 \displaystyle a_1 + \frac{1\kern5em}{\displaystyle
    a_2 +\stackunder{}{\ddots\stackunder{}{\displaystyle
      {}+ \frac{1}{\displaystyle
        a_{n-1} + \frac{1}{a_n}}}
        }}\]
whose associated diagrams when $n=2m$ is even and when $n=2m+1$ is odd appear in Figure \ref{fig:2bdiagram}. If $i$ is odd and $a_i$ is positive or if $i$ is even and $a_i$ is negative, then the twist region associated to $a_i$ consists of crossings of the form $\tikz[baseline=.6ex, scale = .4]{
\draw (0,0) -- (1,1);
\draw (0,1) -- (.3,.7);
\draw (.7,.3) -- (1,0);
}$. If $i$ is odd and $a_i$ is negative or if $i$ is even and $a_i$ is positive, then the twist region associated to $a_i$ consists of crossings of the form $\tikz[baseline=.6ex, scale = .4]{
\draw (0,0) -- (.3,.3);
\draw (.7,.7) -- (1,1);
\draw (0,1) -- (1,0);
}.$ The diagram associated with the continued fraction $(a_1,\dots,a_n)$ is alternating if and only if the signs of all the parameters are the same. We denote the 2-bridge knot associated to the rational number $p/q$ with continued fraction representation $(a_1,\dots,a_n)$ by $K(p/q)$ or $K(a_1,\dots,a_n)$. The 2-bridge knot $K(-p/q)$ is the mirror of the 2-bridge knot $K(p/q)$. Figure \ref{fig:814} shows a specific example of the knot $8_{14}$ which has continued fraction $(2,-4,2,2)$.

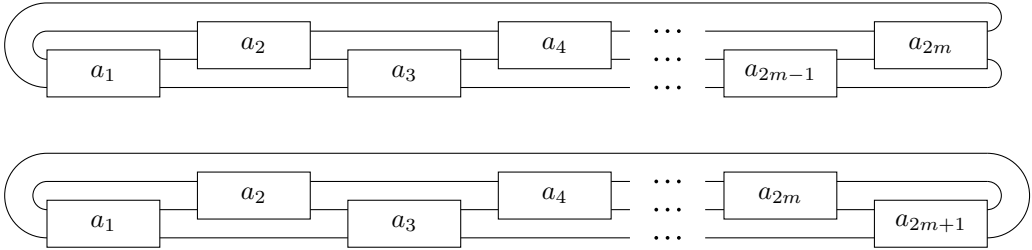
\begin{figure}[h]
\[\begin{tikzpicture}[scale=.5]
\draw (0,0) rectangle (3,1.25);
\draw (1.5, 5/8) node{$a_1$};
\draw (4,.75) rectangle (7,2);
\draw (5.5,11/8) node{$a_2$};

\begin{scope}[xshift=8cm]
\draw (0,0) rectangle (3,1.25);
\draw (1.5, 5/8) node{$a_3$};
\draw (4,.75) rectangle (7,2);
\draw (5.5,11/8) node{$a_4$};
\end{scope}

\begin{scope}[xshift=18cm]
\draw (0,0) rectangle (3,1.25);
\draw (1.5, 5/8) node{$a_{2m-1}$};
\draw (4,.75) rectangle (7,2);
\draw (5.5,11/8) node{$a_{2m}$};

\end{scope}

\fill[black] (16.5,.25) circle (.05);
\fill[black] (16.8,.25) circle (.05);
\fill[black] (16.2,.25) circle (.05);

\fill[black] (16.5,1) circle (.05);
\fill[black] (16.8,1) circle (.05);
\fill[black] (16.2,1) circle (.05);

\fill[black] (16.5,1.75) circle (.05);
\fill[black] (16.8,1.75) circle (.05);
\fill[black] (16.2,1.75) circle (.05);

\draw (3,1) -- (4,1);
\draw (3,.25) -- (8,.25);
\draw (7,1) -- (8,1);
\draw (7,1.75) -- (12,1.75);
\draw (11,1) -- (12,1);
\draw (15,1) -- (15.5,1);
\draw (15,1.75) -- (15.5,1.75);
\draw (11,.25) -- (15.5, .25);
\draw (18,.25) -- (17.5,.25);
\draw (18,1) -- (17.5,1);
\draw (17.5,1.75) -- (22,1.75);
\draw (21,1) -- (22,1);
\draw (21,.25) -- (25,.25);
\draw (25,.25) arc (-90:90:3/8);
\draw (0,1.75) -- (4,1.75);
\draw (0,1.75) arc (90:270:3/8);
\draw (0,2.5) -- (25,2.5);
\draw (0,2.5) arc(90:270:9/8);
\draw (25,2.5) arc(90:-90:3/8);

\begin{scope}[yshift=-4cm]
\draw (0,0) rectangle (3,1.25);
\draw (1.5, 5/8) node{$a_1$};
\draw (4,.75) rectangle (7,2);
\draw (5.5,11/8) node{$a_2$};

\begin{scope}[xshift=8cm]
\draw (0,0) rectangle (3,1.25);
\draw (1.5, 5/8) node{$a_3$};
\draw (4,.75) rectangle (7,2);
\draw (5.5,11/8) node{$a_4$};
\end{scope}

\begin{scope}[xshift=18cm]
\draw (0,.75) rectangle (3,2);
\draw (1.5, 11/8) node{$a_{2m}$};
\draw (4,0) rectangle (7,1.25);
\draw (5.5,5/8) node{$a_{2m+1}$};

\end{scope}

\fill[black] (16.5,.25) circle (.05);
\fill[black] (16.8,.25) circle (.05);
\fill[black] (16.2,.25) circle (.05);

\fill[black] (16.5,1) circle (.05);
\fill[black] (16.8,1) circle (.05);
\fill[black] (16.2,1) circle (.05);

\fill[black] (16.5,1.75) circle (.05);
\fill[black] (16.8,1.75) circle (.05);
\fill[black] (16.2,1.75) circle (.05);

\draw (3,1) -- (4,1);
\draw (3,.25) -- (8,.25);
\draw (7,1) -- (8,1);
\draw (7,1.75) -- (12,1.75);
\draw (11,1) -- (12,1);
\draw (15,1) -- (15.5,1);
\draw (15,1.75) -- (15.5,1.75);
\draw (11,.25) -- (15.5, .25);
\draw (22,.25) -- (17.5,.25);
\draw (18,1) -- (17.5,1);
\draw (17.5,1.75) -- (18,1.75);
\draw (21,1.75) -- (25,1.75);
\draw (21,1) -- (22,1);
\draw (25,1) arc (-90:90:3/8);
\draw (0,1.75) -- (4,1.75);
\draw (0,1.75) arc (90:270:3/8);
\draw (0,2.5) -- (25,2.5);
\draw (0,2.5) arc(90:270:9/8);
\draw (25,2.5) arc(90:-90:9/8);

\end{scope}

\end{tikzpicture}\]
\caption{\textbf{Top.} A 2-bridge knot represented by the tuple $(a_1,a_2,\dots,a_{2m})$. \textbf{Bottom.} The 2-bridge knot represented by the tuple $(a_1,a_2,\dots,a_{2m+1})$.}
\label{fig:2bdiagram}
\end{figure}

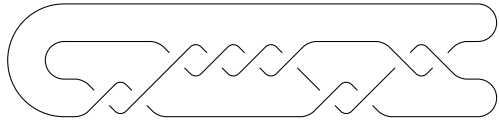
\begin{figure}[h]
\[\begin{tikzpicture}
\begin{scope}[rounded corners = 1mm, scale=.5]
 
\draw (-.5,0) -- (0,0) -- (1,1) -- (1.3,.7);
\draw (.7,.3) -- (1,0) -- (3,2) -- (3.3,1.7);
\draw (1.7,.3) -- (2,0) -- (6,0) -- (7,1) -- (7.3,.7);
\draw (2.7,1.3) -- (3,1) -- (4,2) -- (4.3,1.7);
\draw (3.7,1.3) -- (4,1) -- (5,2) -- (5.3,1.7);
\draw (4.7,1.3) -- (5,1) -- (6,2) -- (8,2) -- (9,1) -- (9.3,1.3);
\draw (5.7,1.3) -- (6.3,.7);
\draw (6.7,.3) -- (7,0) -- (8.3,1.3);
\draw (8.7,1.7) -- (9,2) -- (10,1) -- (10.5,1);
\draw (7.7,.3) -- (8,0) -- (10.5,0);
\draw (.3,.7) -- (0,1) -- (-.5,1);
\draw (2.3,1.7) -- (2,2) -- (-.5,2);
\draw (9.7,1.7) -- (10,2) -- (10.5,2);
\draw (-.5,1) arc (270:90:.5cm);
\draw (10.5,0) arc (-90:90:.5cm);
\draw (-.5,0) arc (270:90:1.5cm);
\draw (10.5,2) arc (-90:90:.5cm);
\draw (-.5,3) -- (10.5,3);

\end{scope}

\end{tikzpicture}\]
\caption{The knot $8_{14}$ is the two bridge knot $K(2,-4,2,2)$.}
\label{fig:814}
\end{figure}

Schubert \cite{Schubert} classified 2-bridge knots as follows.
\begin{theorem}
    The 2-bridge knot $K(p/q)$ and $K(p'/q')$ are equivalent if and only if $p=p'$ and either $q\equiv \pm q'$ mod $p$ or $qq'\equiv \pm 1$ mod $p$.
\end{theorem}

For the remainder of the paper, we will use continued fraction representations of even length, each of whose entries are nonzero even integers.
\begin{definition}
\label{def:even}
A $2m$-tuple $\mathbf{a}=(2a_1,2a_2,\dots,2a_{2m})$ where $a_i\in\mathbb{Z}$ and $a_i\neq 0$ for $1\leq i \leq 2m$ is called an \textit{even continued fraction representation} of the $2$-bridge knot $K(\mathbf{a})$. 
\end{definition} 
The diagram of $8_{14}$ in Figure \ref{fig:814} has even continued fraction representation $(2,-4,2,2)$.  Schubert's classification of $2$-bridge knots leads to the following result on the different possible even continued fraction representations of a 2-bridge knot $K$ or its mirror $\overline{K}$; see Corollary 8.7.3 and Theorem 8.8.1 in Cromwell \cite{Cromwell}. 

\begin{theorem}
\label{thm:even}
Every $2$-bridge knot $K$ has an even continued fraction representation $(2a_1,2a_2,\dots,2a_{2m})$. If a 2-bridge knot $K$ is represented by $(2a_1,2a_2,\dots,2a_{2m})$, then the set of even continued fraction representations of $K$ and its mirror $\overline{K}$ is
\[\{ (2a_1,2a_2,\dots,2a_{2m}),(-2a_{2m},-2a_{2m-1},\dots,-2a_1),(-2a_1,-2a_2,\dots,-2a_{2m}), (2a_{2m},2a_{2m-1},\dots,2a_1)\}.\]
\end{theorem}

It is possible that two of the even continued fractions listed in Theorem \ref{thm:even} are the same $2m$-tuples.
\begin{definition}
    An even continued fraction representation $(2a_1,\dots,2a_{2m})$ is a \textit{palindrome} if $(2a_1,\dots,2a_{2m})=(2a_{2m},\dots,2a_1)$ and an \textit{anti-palindrome} if $(2a_1,\dots,2a_{2m}) = (-2a_{2m},\dots,-2a_1)$.
\end{definition}
 Theorem \ref{thm:even} implies that if a 2-bridge knot is represented by a palindromic or anti-palindromic even continued fraction, then the knot and its mirror have precisely two even continued fraction representations, and if a 2-bridge knot is not represented by a palindromic or anti-palindromic even continued fraction, then the knot and its mirror have precisely four even continued fraction representations.

Our goal is study the braid indices of 2-bridge knots with a fixed crossing number $c$, but the crossing number of the 2-bridge knot $K(2a_1,\dots,2a_{2m})$ is often less than the number of crossings in the diagram in Figure \ref{fig:2bdiagram}. Suzuki \cite{Suzuki} proved the following result on the crossing number of $K(2a_1,\dots,2a_{2m})$ (see also \cite{DEH}).

\begin{theorem}
\label{thm:crossing}
Let $\mathbf{a}=(2a_1,\dots,2a_{2m})$ be an even continued fraction representation of the 2-bridge knot $K(\mathbf{a})$. The crossing number $c(K(\mathbf{a}))$ of $K(\mathbf{a})$ is
\[c(K(\mathbf{a})) = \left(\sum_{i=1}^{2m} 2|a_i|\right) - \ell,\]
where $\ell$ is the number of sign changes in the $2m$-tuple $\mathbf{a}$.
\end{theorem}

Murasugi \cite{Murasugi} showed how to compute the braid index of a 2-bridge knot via its even continued fraction representation.
\begin{theorem}
\label{thm:braid}
    Let $K(\mathbf{a})$ be a 2-bridge knot with even continued fraction representation $\mathbf{a}=(2a_1,\dots,2a_{2m})$. The braid index $b(K(\mathbf{a)})$ of $K(\mathbf{a})$ is
    \[b(K(\mathbf{a})) = \left(\sum_{i=1}^{2m} |a_i|\right) -\ell +1,\]
    where $\ell$ is the number of sign changes in the $2m$-tuple $\mathbf{a}$.
\end{theorem}
Although we do not directly use the result in our proofs, it is interesting to note that Diao, Ernst, and Hetyei \cite{DEH} showed how to compute the braid index of a 2-bridge knot from its alternating diagram. The 2-bridge knot $8_{14}$ knot $K(2,-4,2,2)$ in Figure \ref{fig:2bdiagram} has the even continued fraction representation $(2,-4,2,2)$, which has two sign changes. Therefore the crossing number of $K(2,-4,2,2)$ is eight, and its braid index is four.

\section{The number of 2-bridge knots with crossing number $c$ and braid index $b$}
\label{sec:number}

In this section, we find $k_{c,b}$, the number of $2$-bridge knots with crossing number $c$ and braid index $b$, by studying the set of even continued fraction representations of $2$-bridge knots with crossing number $c$ and the subset of palindromic or anti-palindromic even continued fraction representations. Let $E(c)$ be the set of even continued fraction representations corresponding to $2$-bridge knots of crossing number $c$, that is, let 
\[E(c) =\left\{(2a_1,2a_2,\dots,2a_{2m})~:~ a_i\in\Z, a_i\neq 0, \left(\sum_{i=1}^{2m} 2|a_i|\right) - \ell = c\right\},\]
where $\ell$ is the number of sign changes in the sequence $(2a_1,2a_2,\dots,2a_{2m})$.
Furthermore, let $E(c,b)$ be the subset of $E(c)$ consisting of even continued fraction representations in $E(c)$ corresponding to $2$-bridge knots of braid index $b$, that is, let
\[E(c,b) = \left\{(2a_1,2a_2,\dots,2a_{2m})\in E(c)~:~ \left(\sum_{i=1}^{2m} |a_i|\right) -\ell +1=b\right\},\]
where again $\ell$ is the number of sign changes in the sequence $(2a_1,2a_2,\dots,2a_{2m})$.
Define $e(c)=|E(c)|$ and $e(c,b)=|E(c,b)|$. 

\begin{example} 
\label{ex:e}
All nonempty sets $E(c,b)$ where $c\leq 6$ are as follows:  
    \begin{align*}
        E(3,2) = & \; \{(2,-2), (-2,2)\},\\
        E(4,3) = & \; \{(2,2), (-2,-2)\},\\
        E(5,2) = & \; \{(2,-2,2,-2), (-2,2,-2,2)\},\\
        E(5,3) = & \; \{(2,-4), (-2,4), (4,-2), (-4,2)\},\\
        E(6,3) = & \; \{(2,2,-2,2),(-2,-2,2,-2), (2,-2,2,2), (-2,2,-2,-2), (2,-2,-2,2), (-2,2,2,-2)\},\\
        E(6,4) = & \; \{(2,4), (4,2), (-2,-4), (-4,-2)\}.
    \end{align*}
\end{example}

Let $E_p(c)$ be the subset of $E(c)$ consisting of palindromic or anti-palindromic tuples, and similarly let $E_p(c,b)$ be the subset of $E(c,b)$ consisting of palindromic or anti-palindromic tuples. Define $e_p(c)=|E_p(c)|$ and $e_p(c,b)=|E_p(c,b)|$. 

\begin{example} 
\label{ex:ep}
All nonempty sets $E_p(c,b)$ where $c\leq 6$ are as follows:
\begin{align*}
        E_p(3,2) = & \; \{(2,-2), (-2,2)\},\\
        E_p(4,3) = & \; \{(2,2), (-2,-2)\},\\
        E_p(5,2) = & \; \{(2,-2,2,-2), (-2,2,-2,2)\},\\
        E_p(6,3) = & \; \{(2,-2,-2,2), (-2,2,2,-2)\}.\\
\end{align*}
The sets $E_p(3,2)$ and $E_p(5,2)$ consist of anti-palindromes, while the sets $E_p(4,3)$ and $E_p(6,3)$ consist of palindromes.
\end{example}
Theorem \ref{thm:even} implies that the number $k_{c,b}$ of 2-bridge knots with crossing number $c$ and braid index $b$ is $k_{c,b}=\frac{1}{4}(e(c,b)+e_p(c,b))$ and that the number $|\mathcal{K}_c|$ of $2$-bridge knots with crossing number $c$ is $|\mathcal{K}_c|=\frac{1}{4}(e(c)+e_p(c))$. Ernst and Sumners \cite{ErnSum} first computed $|\mathcal{K}_c|$. In the following proposition, we find recursive and closed formulas for $e(c,b)$.
\begin{proposition}
\label{prop:ecb}
    If $c\geq 6$, then 
    \begin{equation}
    \label{eq:ecbrecur}
    e(c,b) = e(c-2,b) + 2e(c-2,b-1) + 2e(c-3,b-1).
    \end{equation}
    For all $c$ and $b$,
    \[e(c,b) = \begin{cases}
    2 & \text{if $3\leq c$, $c$ is odd, and $b=2$,}\\
    2^{b-2}\binom{c-b}{b-2} &\text{if $3\leq c$ and $3\leq b \leq \left\lceil \frac{c+1}{2}\right\rceil$,}\\
    0 & \text{otherwise.}
\end{cases}\]
\end{proposition}
\begin{proof}
    We first prove the recursive formula and then use the recursive formula to verify the closed formula. In order to prove the recursive formula, we partition $E(c,b)$ into subsets and then show that the sizes of those subsets correspond to the terms in the recursive formula. For the remainder of the proof, let $\mathbf{a}=(2a_1,2a_2,\dots,2a_{2m})$. Define
    \begin{align*}
        E_1(c,b) = & \; \{\mathbf{a}\in E(c,b)~:~|a_{2m}|>1\},\\
        E_2(c,b) = & \; \{\mathbf{a}\in E(c,b)~:~ |a_{2m}|=1,|a_{2m-1}|>1 \},\\   
        E_3(c,b) = & \; \{\mathbf{a}\in E(c,b)~:~ |a_{2m}|=|a_{2m-1}|=1, \sign(a_{2m-2})=\sign(a_{2m-1})=\sign(a_{2m})\},\\
        E_4(c,b) = & \; \{\mathbf{a}\in E(c,b)~:~ |a_{2m}|=|a_{2m-1}|=1, -\sign(a_{2m-2})=\sign(a_{2m-1})=\sign(a_{2m})\},\\
        E_5(c,b) = & \; \{\mathbf{a}\in E(c,b)~:~ |a_{2m}|=|a_{2m-1}|=1, \sign(a_{2m-2})=-\sign(a_{2m-1})=\sign(a_{2m})\},\\
        E_6(c,b) = & \; \{\mathbf{a}\in E(c,b)~:~ |a_{2m}|=|a_{2m-1}|=1, \sign(a_{2m-2})=\sign(a_{2m-1})=-\sign(a_{2m})\}.\\
    \end{align*}
    Also, define $e_i(c,b)=|E_i(c,b)|$ for $1\leq i \leq 6$. Because the subsets $E_i(c,b)$ for $1\leq i \leq 6$ partition $E(c,b)$, it follows that $e(c,b)=\sum_{i=1}^6 e_i(c,b)$. 

    Define $f_1:E_1(c,b)\to E(c-2,b-1)$ by $f_1(\mathbf{a})=(2a_1,\dots, 2a_{2m-1},2a_{2m} -2\sign(a_{2m})).$ Since $\mathbf{a}$ and $f_1(\mathbf{a})$ have the same number of sign changes, Theorems \ref{thm:crossing} and \ref{thm:braid} imply that $f_1(\mathbf{a})\in E(c-2,b-1)$. Because $f_1$ is a bijection, it follows that $e_1(c,b)=e(c-2,b-1)$.

    Define $f_2:E_2(c,b) \to E(c-2,b-1)-E_1(c-2,b-1)$ by \[f_2(\mathbf{a})=(2a_1,\dots,2a_{2m-2},2a_{2m-1}-2\sign(a_{2m-1}),2a_{2m}).\] Since $\mathbf{a}$ and $f_2(\mathbf{a})$ have the same number of sign changes, Theorems \ref{thm:crossing} and \ref{thm:braid} imply that $f_2(\mathbf{a})\in E(c-2,b-1)$. Since the last entry of $f_2(\mathbf{a})$ has absolute value $2$, it follows that $f_2(\mathbf{a})\notin E_1(c-2,b-1)$. Thus $f_2(\mathbf{a})\in E(c-2,b-1)-E_1(c-2,b-1)$. Because $f_2$ is a bijection, it follows that $e_2(c,b)=e(c-2,b-1) - e_1(c-2,b-1)$.

    Define $f_3:E_3(c,b)\to E(c-4,b-2)$ by $f_3(\mathbf{a})=(2a_1,\dots, 2a_{2m-2})$. Since $\mathbf{a}$ and $f_3(\mathbf{a})$ have the same number of sign changes, Theorems \ref{thm:crossing} and \ref{thm:braid} imply that $f_3(\mathbf{a})\in E(c-4,b-2)$. Because $f_3$ is a bijection, it follows that $e_3(c,b)=e(c-4,b-2)$.

    Define $f_4:E_4(c,b)\to E(c-3,b-1)$ by $f_4(\mathbf{a})=(2a_1,\dots, 2a_{2m-2})$. Since $\mathbf{a}$ has one more sign change than $f_4(\mathbf{a})$, Theorems \ref{thm:crossing} and \ref{thm:braid} imply that $f_4(\mathbf{a})\in E(c-3,b-1)$. Because $f_4$ is a bijection, it follows that $e_4(c,b)=e(c-3,b-1)$.

    Define $f_5:E_5(c,b)\to E(c-2,b)$ by $f_5(\mathbf{a})=(2a_1,\dots, 2a_{2m-2})$. Since $\mathbf{a}$ has two more sign changes than $f_5(\mathbf{a})$, Theorems \ref{thm:crossing} and \ref{thm:braid} imply that $f_5(\mathbf{a})\in E(c-2,b)$. Because $f_5$ is a bijection, it follows that $e_5(c,b)=e(c-2,b)$.

    Define $f_6:E_6(c,b)\to E(c-3,b-1)$ by $f_6(\mathbf{a})=(2a_1,\dots, 2a_{2m-2})$. Since $\mathbf{a}$ has one more sign change than $f_6(\mathbf{a})$, Theorems \ref{thm:crossing} and \ref{thm:braid} imply that $f_6(\mathbf{a})\in E(c-3,b-1)$. Because $f_6$ is a bijection, it follows that $e_6(c,b)=e(c-3,b-1)$.

    Therefore
    \begin{align*}
        e(c,b) = & \; \sum_{i=1}^6 e_i(c,b)\\
        = & \; 2e(c-2,b-1) -e_1(c-2,b-1) + e(c-4,b-2)+ 2e(c-3,b-1)+e(c-2,b)\\
        = & \; 2e(c-2,b-1) - e(c-4,b-2) + e(c-4,b-2) + 2e(c-3,b-1)+e(c-2,b)\\
        = & \; e(c-2,b) + 2e(c-2,b-1) + 2e(c-3,b-1),
    \end{align*}
    proving the recursive formula.

    In order to prove the closed formula, we first observe that Example \ref{ex:e} confirms that closed formula when $c\le 6$. Also, the unknot is the only knot whose braid index is one, and thus $e(c,1)=0$ for all $c$.

If $c\geq 6$ proceed via induction on $c$. Suppose that $b=2$. Since $e(c,1)=0$ for all $c$, Equation \eqref{eq:ecbrecur} implies that $e(c,2)= e(c-2,2)$ for all $c\geq 6$. Since $e(3,2) = 2$ and $e(4,2)=0$, it follows that $e(c,2) = 2$ when $3\leq c$ and $c$ is odd, and $e(c,2)=0$ when $3\leq c$ and $c$ is even. Suppose that $6\leq c$ and $\left\lceil\frac{c+1}{2}\right\rceil<b$. By the inductive hypothesis, $e(c-2,b)$, $e(c-2,b-1)$ and $e(c-3,b-1)$ are all zero, and hence Equation \eqref{eq:ecbrecur} implies that $e(c,b)$ is also zero.

Suppose that $6\leq c$ and $b=3$. Equation \eqref{eq:ecbrecur} implies that
\[e(c,3) = e(c-2,3) +2e(c-2,2) +2e(c-3,2) = 2\binom{c-5}{1}+4 = 2c-6 = 2\binom{c-3}{1},\]
as desired.

Finally, suppose that $6\leq c$ and $3< b \leq \left\lceil\frac{c+1}{2}\right\rceil$. The inductive hypothesis implies that
\begin{align*}
    e(c,b) = & \; e(c-2,b) + 2e(c-2,b-1) + 2e(c-3,b-1)\\
     = & \; 2^{b-2}\binom{c-b-2}{b-2} + 2\left(2^{b-3} \binom{c-b-1}{b-3}\right) + 2\left(2^{b-3}\binom{c-b-2}{b-3}\right)\\
     = & \; 2^{b-2}\left( \binom{c-b-2}{b-2} + \binom{c-b-1}{b-3} + \binom{c-b-2}{b-3}\right)\\
    = &\; 2^{b-2}\left(\binom{c-b-1}{b-3} + \binom{c-b-1}{b-2}\right)\\
    =&\; 2^{b-2}\binom{c-b}{b-2}.
\end{align*}
\end{proof}
As a corollary, we compute the number $e(c)$ of elements of the set $E(c)$.
\begin{corollary}
\label{cor:ec}
    For each $c\geq 3$, the number of elements of $E(c)$ is $e(c)=\frac{2\left(2^{c-2}-(-1)^{c-2}\right)}{3}.$
\end{corollary}
\begin{proof}
Equation \eqref{eq:ecbrecur} implies 
\begin{align*}
    e(c) = &\; \sum_{b=2}^{\left\lceil \frac{c+1}{2}\right\rceil} e(c,b)\\
    = & \;\sum_{b=2}^{\left\lceil \frac{c+1}{2}\right\rceil} e(c-2,b) +  2 \sum_{b=2}^{\left\lceil \frac{c+1}{2}\right\rceil} e(c-2,b-1) + 2\sum_{b=2}^{\left\lceil \frac{c+1}{2}\right\rceil} e(c-3,b-1)\\
    = & \; 3e(c-2) + 2e(c-3).
    \end{align*}
    Example \ref{ex:e} confirms the desired formula for $c=3$, $4$, and $5$. The result follows by induction since if $c\geq 6$, then
\begin{align*}
    e(c) = & \;3e(c-2) + 2e(c-3)\\
    = & \; 3\left(\frac{2\left(2^{c-4}-(-1)^{c-4}\right)}{3}\right) + 2\left(\frac{2\left(2^{c-5}-(-1)^{c-5}\right)}{3}\right)\\
    = & \; \frac{3 \left(2^{c-3}\right)-6(-1)^{c} + 2^{c-3} +4(-1)^{c}}{3}\\
    =& \; \frac{2\left(2^{c-2}-(-1)^{c-2}\right)}{3}.\\
\end{align*}
\end{proof}

In Table \ref{table:ecb}, we use Proposition \ref{prop:ecb} and Corollary \ref{cor:ec}  to compute $e(c,b)$  and $e(c)$ for small values of $c$ and $b$.
\begin{table}[h]
\label{table:ecb}
    \centering
 \begin{tabular}{||c | c c c c c c || c||} 
 \hline
 $c \backslash b$ & 2 & 3 & 4 & 5 & 6 & 7 & $e(c)$\\ [0.5ex] 
 \hline
 3 & 2 &   &   &    &   &   & 2 \\ 
 4 &    & 2 &   &   &   &   & 2  \\
 5 & 2 & 4 &   &   &   &   & 6\\
 6 &   & 6 & 4 &   &   &   & 10\\
 7 & 2 & 8 & 12 &   &   &   & 22\\ 
 8 &   & 10 & 24 & 8 &   &   & 42\\
 9 & 2 & 12 & 40 & 32 &   &   & 86\\
 10 &   & 14 & 60 & 80 & 16 &   & 170\\
 11 & 2 & 16 & 84 & 140 & 80 &   & 342\\
 12 &   & 18 & 112 & 280 & 240 & 32 & 682\\[1ex]
 \hline
 \end{tabular} 
 \caption{The table shows $e(c,b)$ and $e(c)$ for small values of the crossing number $c$ and the braid index $b$. Empty entries are $0$.}
\end{table}

We now find recursive and closed formulas for $e_p(c,b)$, the number of palindromic or anti-palindromic even continued fraction representations with crossing number $c$ and braid index $b$.
\begin{proposition}
    \label{prop:ep}
    If $c\geq 7$, then 
    \begin{equation}
    \label{eq:ep}
        e_p(c,b) = e_p(c-2,b) + 2e_p(c-4,b-2).
    \end{equation}
    For all $c$ and $b$,
    \[e_p(c,b) = \begin{cases}
    2^{\frac{b-1}{2}}\binom{\frac{c-b-1}{2}}{\frac{b-3}{2}}& \text{if $2 \leq b \leq \left\lceil\frac{c+1}{2}\right\rceil$, $c$ is even, and $b$ is odd,}\\
    2^{\frac{b}{2}}\binom{\frac{c-b-1}{2}}{\frac{b-2}{2}}&
    \text{if $2\leq b \leq \left\lceil\frac{c+1}{2}\right\rceil$, $c$ is odd, and $b$ is even,}\\
    0 & \text{otherwise.}
\end{cases}\]
\end{proposition}
\begin{proof}
As in the proof of Proposition \ref{prop:ecb}, we first prove the recursive formula and then use the recursive formula to verify the closed formula. In order to prove the recursive formula, we partition $E_p(c,b)$ into subsets and then show that the sizes of those subsets correspond to the terms in the recursive formula. For the remainder of the proof, let $\mathbf{a}=(2a_1,2a_2,\dots,2a_{2m})$. Define
\begin{align*}
    E_{p,1}(c,b) = & \; \{\mathbf{a}\in E_p(c,b)~:~|a_1|=|a_{2m}| >1\},\\
    E_{p,2}(c,b) = & \; \{\mathbf{a}\in E_p(c,b)~:~|a_1|=|a_{2m}|=1, \sign(a_1)=\sign(a_2)\},\\
    E_{p,3}(c,b) = & \; \{\mathbf{a}\in E_p(c,b)~:~|a_1|=|a_{2m}|=1, \sign(a_1) = -\sign(a_2)\}.
\end{align*}
By the definition of $E_p(c,b)$, if $\mathbf{a}\in E_{p,2}(c,b)$, then  $\sign(a_{2m-1})=\sign(a_{2m})$, and if $\mathbf{a}\in E_{p,3}(c,b)$, then $\sign(a_{2m-1})=-\sign(a_{2m})$. Let $e_{p,i}(c,b)=|E_{p,i}(c,b)|$ for $i=1$, $2$, and $3$. Since the subsets $E_{p,i}(c,b)$ for $1\leq i \leq 3$ partition $E_p(c,b)$, it follows that $e_p(c,b) = \sum_{i=1}^3 e_{p,i}(c,b)$.

Define $f_1:E_{p,1}(c,b) \to E_p(c-4,b-2)$ by
\[f_1(\mathbf{a}) = (2a_1-2\sign(a_1),2a_2,\dots,2a_{2m-1}, 2a_{2m} -2\sign(a_{2m})).\]
Since $\mathbf{a}$ and $f_1(\mathbf{a})$ have the same number of sign changes, Theorems \ref{thm:crossing} and \ref{thm:braid} imply that $f_1(\mathbf{a})\in E_p(c-4,b-2)$. Because $f_1$ is a bijection, it follows that $e_{p,1}(c,b) = e_p(c-4,b-2)$.

Define $f_2:E_{p,2}(c,b)\to E_p(c-4,b-2) $ by $f_2(\mathbf{a})=(2a_2,2a_3,\dots,2a_{2m-2},2a_{2m-1})$. Since $\mathbf{a}$ and $f_2(\mathbf{a})$ have the same number of sign changes, Theorems \ref{thm:crossing} and \ref{thm:braid} imply that $f_2(\mathbf{a})\in E_p(c-4,b-2)$. Because $f_2$ is a bijection, it follows that $e_{p,2}(c,b)=e_p(c-4,b-2)$. 

Define $f_3:E_{p,3}(c,b)\to E_p(c-2,b) $ by $f_3(\mathbf{a})=(2a_2,2a_3,\dots,2a_{2m-2},2a_{2m-1})$. Since $\mathbf{a}$ has two more sign changes than $f_3(\mathbf{a})$, Theorems \ref{thm:crossing} and \ref{thm:braid} imply that $f_3(\mathbf{a})\in E_p(c-2,b)$. Because $f_3$ is a bijection, it follows that $e_{p,3}(c,b)=e_p(c-2,b)$. 

Therefore
\[e_p(c,b) = \sum_{i=1}^3 e_{p,i}(c,b) = e_p(c-2,b) + 2e_p(c-4,b-2),\]
which proves Equation \eqref{eq:ep}. 

In order to prove the closed formula for $e_p(c,b)$, we first observe that Example \ref{ex:ep} confirms the formula when $c<7$. Now proceed by induction on $c$. Since $e(c,0)=e(c,1)=0$ for all $c$, it follows that $e_p(c,0)=e_p(c,1)=0$ for all $c$. 

Suppose that $c\geq 7$, $2\leq b \leq \left\lceil\frac{c+1}{2}\right\rceil$, $c$ is even, and $b$ is odd. Then Equation \eqref{eq:ep} implies
    \begin{align*}
        e_p(c,b) = & \; e_p(c-2,b) + 2e_p(c-4,b-2)\\
         = & \; 2^{\frac{b-1}{2}}\binom{\frac{c-b-3}{2}}{\frac{b-3}{2}}+2\left(2^{\frac{b-3}{2}}\right)\binom{\frac{c-b-3}{2}}{\frac{b-5}{2}}\\
         = & \; 2^{\frac{b-1}{2}}\left(\binom{\frac{c-b-3}{2}}{\frac{b-3}{2}} + \binom{\frac{c-b-3}{2}}{\frac{b-5}{2}}\right)\\
         = & \; 2^{\frac{b-1}{2}}\binom{\frac{c-b-1}{2}}{\frac{b-3}{2}}.
    \end{align*}
Suppose that $c\geq 7$, $2\leq b \leq \left\lceil\frac{c+1}{2}\right\rceil$, $c$ is odd, and $b$ is even. Equation \eqref{eq:ep} implies
\begin{align*}
    e_p(c,b) = & \; e_p(c-2,b) + 2e_p(c-4,b-2)\\
    = & \; 2^{\frac{b}{2}}\binom{\frac{c-b-3}{2}}{\frac{b-2}{2}}+2 \left(2^{\frac{b-2}{2}}\right) \binom{\frac{c-b-3}{2}}{\frac{b-4}{2}}\\
    = & \; 2^{\frac{b}{2}}\left(\binom{\frac{c-b-3}{2}}{\frac{b-2}{2}} + \binom{\frac{c-b-3}{2}}{\frac{b-4}{2}} \right)\\
    = & \; 2^{\frac{b}{2}}\binom{\frac{c-b-1}{2}}{\frac{b-2}{2}}.
\end{align*}
For all other values of $c$ and $b$, both terms $e_p(c-2,b)$ and $e_p(c-4,b-2)$ are $0$, and Equation \eqref{eq:ep} implies the result.  
\end{proof}

As a corollary, we compute the number $e_p(c)$ of elements of the set $E_p(c)$.
\begin{corollary}
\label{cor:ep}
    For each $c\geq 3$, the number of elements of $E_p(c)$ is 
    \[e_p(c) = 
    \frac{2\left(2^{\left\lfloor\frac{c-1}{2}\right\rfloor}-(-1)^{\left\lfloor\frac{c-1}{2}\right\rfloor}\right)}{3}. 
\]
\end{corollary}
\begin{proof}
    Equation \eqref{eq:ep} implies
    \[e_p(c) = \sum_{b=2}^{\left\lceil\frac{c+1}{2}\right\rceil} e_p(c,b) = \sum_{b=2}^{\left\lceil\frac{c+1}{2}\right\rceil} e_p(c-2,b)+2e_p(c-4,b-2) = e_p(c-2)+2e_p(c-4). \]
    Example \ref{ex:ep} shows that the desired formula for $e_p(c)$ holds when $3\leq c \leq 6$. When $c\geq 7$, the desired result follows via induction because
    \begin{align*}
        e_p(c) = & \; e_p(c-2) + 2e_p(c-4)\\
        = & \; \frac{2\left(2^{\left\lfloor\frac{c-3}{2}\right\rfloor}-(-1)^{\left\lfloor\frac{c-3}{2}\right\rfloor}\right)}{3} + 2\left(\frac{2\left(2^{\left\lfloor\frac{c-5}{2}\right\rfloor}-(-1)^{\left\lfloor\frac{c-5}{2}\right\rfloor}\right)}{3}\right)\\
        = & \; \frac{2^{\left\lfloor\frac{c-1}{2}\right\rfloor}+2(-1)^{\left\lfloor\frac{c-1}{2}\right\rfloor} + 2^{\left\lfloor\frac{c-1}{2}\right\rfloor} - 4(-1)^{\left\lfloor\frac{c-1}{2}\right\rfloor}}{3}\\
        = & \;   \frac{2\left(2^{\left\lfloor\frac{c-1}{2}\right\rfloor}-(-1)^{\left\lfloor\frac{c-1}{2}\right\rfloor}\right)}{3}. 
    \end{align*}
\end{proof}

In Table \ref{table:epcb}, we use Proposition \ref{prop:ep} and Corollary \ref{cor:ep} to compute $e_p(c,b)$ and $e_p(c)$ for small values of $c$ and $b$. 
\begin{table}[h!]
\label{table:epcb}
\centering
 \begin{tabular}{|| c | c c c c c c || c ||} 
 \hline
  $c \backslash b$ & 2 & 3 & 4 & 5 & 6 & 7 & $e_p(c)$ \\ [0.5ex] 
 \hline
 3 & 2 &&&&&& 2\\
 4 & & 2 &&&&& 2\\
 5 & 2 &&&&& & 2\\
 6 & & 2 &&&& & 2\\
 7 & 2 & & 4 &&& & 6\\
 8 & & 2 & & 4 && & 6\\
 9 & 2 & & 8 &&& & 10\\
 10 & & 2 & & 8 && & 10\\
 11 & 2 & & 12 & & 8 & & 22\\
 12 & & 2 & & 12 & & 8 & 22 \\
 \hline
 \end{tabular} 
 \caption{The table shows $e_p(c,b)$ and $e_p(c)$ for small values of the crossing number $c$ and the braid index $b$. Empty entries are 0.}
\end{table}
Propositions \ref{prop:ecb} and \ref{prop:ep} allow us to prove Theorem \ref{thm:number}.
\begin{proof}[Proof of Theorem \ref{thm:number}]
Theorem \ref{thm:even} implies each $2$-bridge knot with crossing number $c$ is represented by either two or four elements in $E(c)$ depending on whether the corresponding elements are in $E_p(c)$ or not, respectively. Therefore the number $k_{c,b}$ of $2$-bridge knots with crossing number $c$ and braid index $b$ is $\frac{1}{4}(e(c,b) + e_p(c,b))$. Propositions \ref{prop:ecb} and \ref{prop:ep} imply the result.
\end{proof} 

Table \ref{table:kcb} shows the number $k_{c,b}$ of 2-bridge knots with crossing number $c$ and braid index $b$ for $3\leq c \leq 20$.
\begin{table}[h]
    \label{table:kcb}
    \begin{tabular}{||c | c c c c c c c c c c||}
    \hline
    $c \backslash b$ & 2 & 3 & 4 & 5 & 6 & 7 & 8 & 9 & 10 & 11\\
    \hline
    \hline
    3 & 1 & & & & & & & & &    \\
    4 &  & 1 & & & & & & & &    \\
    5 & 1 & 1 & & & & & & & &    \\
    6 &  & 2 & 1 & & & & & & &    \\
    7 & 1 & 2 & 4 & & & & & & &    \\
    8 & & 3 & 6 & 3 & & & & & &    \\
    9 & 1 & 3 & 12 & 8 & & & & & &    \\
    10 & & 4 & 15 & 22 & 4 & & & & &    \\
    11 & 1 & 4 & 24 & 40 & 22 & & & & &    \\
    12 & & 5 & 28 & 73 & 60 & 10 & & & &    \\
    13 & 1 & 5 & 40 & 112 & 146 & 48 & & & &    \\
    14 & & 6 & 45 & 172 & 280 & 174 & 16 & & &    \\
    15 & 1 & 6 & 60 & 240 & 516 & 448 & 116 & & &    \\
    16 & & 7 & 66 & 335 & 840 & 1,020 & 448 & 36 & &    \\
    17 & 1 & 7 & 84 & 440 & 1,340 & 2,016 & 1,360 & 256 & &   \\
    18 & & 8 & 91 & 578 & 1,980 & 3,716 & 3,360 & 1,168 & 64 &    \\
    19 & 1 & 8 & 112 & 728 & 2,890 & 6,336 & 7,432 & 3,840 & 584 &    \\
    20 & & 9 & 120 & 917 & 4,004 & 10,326 & 14,784 & 10,600 & 2,880 & 136    \\
    \hline
    \end{tabular}
    \caption{The number $k_{c,b}$ of 2-bridge knots with crossing number $c$ and braid index $b$. Empty entries are 0.}
\end{table}

\section{The mode of the braid indices of 2-bridge knots}
\label{sec:mode}
In this section, we find the mode of the braid indices of 2-bridge knots of a fixed crossing number $c$, proving Theorem \ref{thm:mode}. We begin by proving two properties of the finite sequence $(e(c,b))_{b=2}^n$ where $n={\left\lceil\frac{c+1}{2}\right\rceil}$.

A finite sequence $(a_1,\dots,a_n)$ is \textit{log concave} if $a_i^2\geq a_{i-1}a_{i+1}$ for $1<i<n$ and \textit{unimodal} if there is an index $m$ such that $a_1\leq a_2\leq\cdots\leq a_{m-1}\leq a_m \geq a_{m+1}\geq \cdots \geq a_{n-1}\geq a_n$. Log concave sequences are known to be unimodal. Our strategy to prove Theorem \ref{thm:mode} is to first show that the sequence $(e(c,b))_{b=2}^n$ where $n={\left\lceil\frac{c+1}{2}\right\rceil}$ is log concave with mode $b = \left\lceil\frac{c}{3}\right\rceil+1$, and then show that $e_p(c,b)$ is small enough in comparison to $e(c,b)$ so that the mode of the finite sequence $(k_{c,b})_{b=2}^n=\left(\frac{1}{4}(e(c,b)+e_p(c,b))\right)_{b=2}^n$ is also $b = \left\lceil\frac{c}{3}\right\rceil+1$.

Any sequence of length two or shorter is log concave. The first time the sequence of $e(c,b)$ for a fixed $c$ has length at least three is when $c=7$.
\begin{lemma}
\label{lem:logconcave}
    For each $c\geq 7$, the sequence $(e(c,b))_{b=2}^n$ where $n={\left\lceil\frac{c+1}{2}\right\rceil}$ is log concave.
\end{lemma}
\begin{proof}
    If $b=3$ and $c$ is even, then $e(c,3)^2\geq 0 = e(c,2)e(c,4)$. If $b=3$ and $c$ is odd, then 
    \begin{align*}
    e(c,3)^2 = & \; \left(2\binom{c-3}{1}\right)^2 = 4 c^2 - 24 c + 36~\text{and}\\
    e(c,2)e(c,4) =  &\; 2\cdot 2^2\binom{c-4}{2} = 4c^2-36c+80.
    \end{align*}
    Since $e(c,3)^2=4c^2-24c+36\geq 4c^2-36c+80$ when $c\geq 4$, the result follows for $b=3$.

    Let $b>3$. Then
    \begin{align*}
        \frac{e(c,b)^2}{e(c,b-1)e(c,b+1)} = & \; \frac{2^{2b-4}\binom{c-b}{b-2}^2}{2^{2b-4}\binom{c-b+1}{b-3}\binom{c-b-1}{b-1}}\\
        = & \; \frac{(c-b)!^2(b-3)!(c-2b+4)!(b-1)!(c-2b)!}{(c-b+1)!(c-b-1)!(b-2)!^2 (c-2b+2)!^2}\\
        = & \;\left(\frac{(c-b)(b-1)}{(c-b+1)(b-2)}\right)\left(\frac{(c-2b+4)(c-2b+3)}{(c-2b+2)(c-2b+1)}\right).\\
    \end{align*}
    The second fraction in the above product is clearly greater than one. Next, we show that the first fraction in the above product is at least one. Since $b\le n$, it follows that $b\leq\frac{c+2}{2}$ or $2b-2\le c$. Therefore $3b-2c-2\leq b-c$, and it follows that
    \[(c-b+1)(b-2)= cb-2c-b^2+3b-2 \leq cb-c-b^2+b = (c-b)(b-1).\]
    Thus the quotient $\frac{e(c,b)^2}{e(c,b-1)e(c,b+1)}$ is at least one, and therefore the sequence $(e(c,b))_{b=2}^n$ is log concave.
\end{proof}

The next lemma gives the mode of the sequence $(e(c,b))_{b=2}^n$ where $n={\left\lceil\frac{c+1}{2}\right\rceil}$.
\begin{lemma}
\label{lem:emode}
    For each $c\geq 7$, the largest term in the sequence $(e(c,b))_{b=2}^n$ where $n={\left\lceil\frac{c+1}{2}\right\rceil}$ is $e\left(c,\left\lceil\frac{c}{3}\right\rceil+1\right)$.
\end{lemma}
\begin{proof}
    Since the sequence $(e(c,b))_{b=2}^n$ where $n={\left\lceil\frac{c+1}{2}\right\rceil}$ is unimodal, it suffices the show that
    \[e\left(c,\left\lceil\frac{c}{3}\right\rceil\right)\leq e\left(c,\left\lceil\frac{c}{3}\right\rceil+1\right)\geq e\left(c,\left\lceil\frac{c}{3}\right\rceil+2\right).\]
    When $c\geq 8$, Proposition \ref{prop:ecb} implies the difference $e\left(c,\left\lceil\frac{c}{3}\right\rceil+1\right)-e\left(c,\left\lceil\frac{c}{3}\right\rceil\right)$ is given by
    \begin{equation}
    \label{eq:dif1}    
    e\left(c,\left\lceil\frac{c}{3}\right\rceil+1\right)-e\left(c,\left\lceil\frac{c}{3}\right\rceil\right) = \begin{cases}
        \displaystyle \frac{2^{r-1}}{r}\binom{2r+1}{r-1} & \text{if $c=3r$,}\\
        \displaystyle \frac{2^{r-1}}{r}\binom{2r}{r-1}& \text{if $c=3r+1$,}\\
        \displaystyle \frac{2^{r-1}(5r+4)}{(r+1)(r+2)}\binom{2r}{r} & \text{if $c=3r+2$.}
    \end{cases}
    \end{equation}
    When $c\geq 8$, Proposition \ref{prop:ecb} also implies the difference $e\left(c,\left\lceil\frac{c}{3}\right\rceil+1\right)-e\left(c,\left\lceil\frac{c}{3}\right\rceil+2\right)$ is given by
    \begin{equation}
    \label{eq:dif2}    
    e\left(c,\left\lceil\frac{c}{3}\right\rceil+1\right)-e\left(c,\left\lceil\frac{c}{3}\right\rceil+2\right) = 
    \begin{cases}
        \displaystyle \frac{2^{r-1}}{r}\binom{2r-2}{r-1}&\text{if $c=3r$,}\\
        \displaystyle \frac{2^{r}(7r-5)}{r(r+1)}\binom{2r-2}{r-1}&\text{if $c=3r+1$,}\\
        \displaystyle \frac{2^{r+2}}{r+1}\binom{2r-1}{r-1}&\text{if $c=3r+2$.}
    \end{cases}
    \end{equation}
    Since both $e\left(c,\left\lceil\frac{c}{3}\right\rceil+1\right)-e\left(c,\left\lceil\frac{c}{3}\right\rceil\right)$ and  $e\left(c,\left\lceil\frac{c}{3}\right\rceil+1\right)-e\left(c,\left\lceil\frac{c}{3}\right\rceil+2\right)$ are positive, the result follows.
\end{proof}

We use Lemmas \ref{lem:logconcave} and \ref{lem:emode} to prove Theorem \ref{thm:mode}. The inequality
\begin{equation}
    \label{eq:binom}
    \left(\frac{n}{k}\right)^k \leq \binom{n}{k}
\end{equation}
for $0\leq k \leq n$ will be useful in our proof.

\begin{proof}[Proof of Theorem \ref{thm:mode}] Lemmas \ref{lem:logconcave} and \ref{lem:emode} imply that the sequence $(e(c,b))_{b=2}^n$ where $n={\left\lceil\frac{c+1}{2}\right\rceil}$ is unimodal with mode  $e\left(c,\left\lceil\frac{c}{3}\right\rceil+1\right)$. Since $k_{c,b}=\frac{1}{4}(e(c,b)+e_p(c,b))$, to show that $k_{c,\left\lceil c/3\right\rceil+1}\geq k_{c,b}$ for $2\leq b \leq \left\lceil\frac{c+1}{2}\right\rceil$, it suffices to show that $e(c,\left\lceil\frac{c}{3}\right\rceil+1)+e_p(c,\left\lceil\frac{c}{3}\right\rceil+1) \geq e(c,b)+e_p(c,b)$ for $2\leq b \leq \left\lceil\frac{c+1}{2}\right\rceil.$ Since
\[e\left(c,\left\lceil\frac{c}{3}\right\rceil+1\right) - e(c,b) \geq \min\left\{ e\left(c,\left\lceil\frac{c}{3}\right\rceil+1\right) - e\left(c,\left\lceil\frac{c}{3}\right\rceil\right),e\left(c,\left\lceil\frac{c}{3}\right\rceil+1\right)-e\left(c,\left\lceil\frac{c}{3}\right\rceil+2\right)\right\}\]
and $e_p(c,b)-e_p\left(c,\left\lceil\frac{c}{3}\right\rceil+1\right)\leq e_p(c)$ for $2\leq b \leq \left\lceil\frac{c+1}{2}\right\rceil$, it suffices to show that 
\begin{align}
\label{eq:mode1}    
e_p(c)\leq & \; e\left(c,\left\lceil\frac{c}{3}\right\rceil+1\right) - e\left(c,\left\lceil\frac{c}{3}\right\rceil\right)~\text{and}\\
\label{eq:mode2}
e_p(c) \leq & \;  e\left(c,\left\lceil\frac{c}{3}\right\rceil+1\right) - e\left(c,\left\lceil\frac{c}{3}\right\rceil+2\right).
\end{align}
We use Equations \eqref{eq:dif1} and \eqref{eq:dif2} to show that Inequalities \eqref{eq:mode1} and \eqref{eq:mode2} hold for $c\geq 34$. For $c\leq 34$, we use Theorem \ref{thm:number} to check that $k_{c,\left\lceil c/3\right\rceil+1}\geq k_{c,b}$ for $2\leq b \leq \left\lceil\frac{c+1}{2}\right\rceil.$

Suppose that $c=3r$. Corollary \ref{cor:ep} implies that $e_p(c) \leq \frac{1}{3}\left(2^{\frac{c+2}{2}}\right) = \frac{1}{3}\left(2^{\frac{3r+2}{2}}\right)$. Equation \eqref{eq:dif1} and Inequality \eqref{eq:binom} imply that
\[
   e\left(c,\left\lceil\frac{c}{3}\right\rceil+1\right)-e\left(c,\left\lceil\frac{c}{3}\right\rceil\right) = \frac{2^{r-1}}{r}\binom{2r+1}{r-1}
   \geq  \left(\frac{2^{r-1}}{r}\right) \left(\frac{2r+1}{r-1}\right)^{r-1}
   \geq  \frac{2^{2r-2}}{r}.
\]
Since $\frac{2^{2r-2}}{r} \geq \frac{1}{3}\left(2^{\frac{3r+2}{2}}\right)$ when $r\geq 10$, Inequality \eqref{eq:mode1} holds with $c=3r\geq 30$. Equation \eqref{eq:dif2} and Inequality \eqref{eq:binom} imply that  
\[e\left(c,\left\lceil\frac{c}{3}\right\rceil+1\right)-e\left(c,\left\lceil\frac{c}{3}\right\rceil+2\right) = \frac{2^{r-1}}{r}\binom{2r-2}{r-1} \geq \left(\frac{2^{r-1}}{r}\right)\left(\frac{2r-2}{r-1}\right)^{r-1} = \frac{2^{2r-2}}{r}.\]
Since $\frac{2^{2r-2}}{r} \geq \frac{1}{3}\left(2^{\frac{3r+2}{2}}\right)$ when $r\geq 10$, Inequality \eqref{eq:mode2} holds with $c=3r\geq 30$.

Suppose that $c=3r+1$. Corollary \ref{cor:ep} implies that $e_p(c)\leq \frac{1}{3}\left(2^{\frac{c+2}{2}}\right)=\frac{1}{3}\left(2^{\frac{3r+3}{2}}\right).$  Equation \eqref{eq:dif1} and Inequality \eqref{eq:binom} imply that  
\[e\left(c,\left\lceil\frac{c}{3}\right\rceil+1\right)-e\left(c,\left\lceil\frac{c}{3}\right\rceil\right) = \frac{2^{r-1}}{r}\binom{2r}{r-1} \geq \left(\frac{2^{r-1}}{r}\right) \left(\frac{2r}{r-1}\right)^{r-1}\geq \frac{2^{2r-2}}{r}.\]
Since $\frac{2^{2r-2}}{r}\geq\frac{1}{3}\left(2^{\frac{3r+3}{2}}\right)$ when $r\geq 11$, Inequality \eqref{eq:mode1} holds with $c=3r+1\geq 34$. Equation \eqref{eq:dif2} and Inequality \eqref{eq:binom} imply that  
\[e\left(c,\left\lceil\frac{c}{3}\right\rceil+1\right)-e\left(c,\left\lceil\frac{c}{3}\right\rceil+2\right) =  \frac{2^{r}(7r-5)}{r(r+1)}\binom{2r-2}{r-1} \geq \left(
\frac{2^{r}(7r-5)}{r(r+1)}\right)\left(\frac{2r-2}{r-1}\right)^{r-1}  = \frac{2^{2r-1}(7r-5)}{r(r+1)}.\]
Since $\frac{2^{2r-1}(7r-5)}{r(r+1)}\geq \frac{1}{3}\left(2^{\frac{3r+3}{2}}\right)$ when $r\geq 4$, Inequality \eqref{eq:mode2} holds with $c=3r+1\geq 13$.

Suppose that $c=3r+2$.  Corollary \ref{cor:ep} implies that $e_p(c)\leq \frac{1}{3}\left(2^{\frac{c+2}{2}}\right)=\frac{1}{3}\left(2^{\frac{3r+4}{2}}\right).$ Equation \eqref{eq:dif1} and Inequality \eqref{eq:binom} imply that  
\[e\left(c,\left\lceil\frac{c}{3}\right\rceil+1\right)-e\left(c,\left\lceil\frac{c}{3}\right\rceil\right) = \frac{2^{r-1}(5r+4)}{(r+1)(r+2)}\binom{2r}{r} \geq \left( \frac{2^{r-1}(5r+4)}{(r+1)(r+2)}\right) \left( \frac{2r}{r}\right)^r =  \frac{2^{2r-1}(5r+4)}{(r+1)(r+2)}.\]
Since $\frac{2^{2r-1}(5r+4)}{(r+1)(r+2)} \geq \frac{1}{3}\left(2^{\frac{3r+4}{2}}\right)$ when $r\geq 3$. Inequality \eqref{eq:mode1} holds with $c=3r+2\geq 11$. Equation \eqref{eq:dif2} and Inequality \eqref{eq:binom} imply that  
\[e\left(c,\left\lceil\frac{c}{3}\right\rceil+1\right)-e\left(c,\left\lceil\frac{c}{3}\right\rceil+2\right) =  \frac{2^{r+2}}{r+1}\binom{2r-1}{r-1} \geq \left(\frac{2^{r+2}}{r+1}\right)\left(\frac{2r-1}{r-1}\right)^r\geq \frac{2^{2r+1}}{r+1}.\]
Since $ \frac{2^{2r+1}}{r+1} \geq \frac{1}{3}\left(2^{\frac{3r+4}{2}}\right)$ when $r\geq 2$. Inequality \eqref{eq:mode2} holds with $c=3r+2\geq 8$.

Therefore if $c\geq 34$, both Inequality \eqref{eq:mode1} and Inequality \eqref{eq:mode2} hold, and the desired result follows. The result follows for $c<34$ by direct computation of $k_{c,b}$ using Theorem \ref{thm:number}.
\end{proof}

Using better bounds in the above proof can eliminate the need to use Theorem \ref{thm:number} and direct computation for many values of $c$. However, the proof using better bounds is significantly longer. The quantity $e_p(c)$ will typically be much larger than $e_p(c,b)-e_p\left(c,\left\lceil\frac{c}{3}\right\rceil+1\right)$, and the upper bound $e_p(c)$ could be replaced by $e_p(c,m_c)$ where $m_c$ is the mode of the sequence $(e_p(c,b))_{b=2}^{n}$. For the sake of brevity, we forego this strategy.

The \text{median} of a finite sequence $(a_1,\dots, a_n)$ of non-negative integers is the index $m$ such that $\sum_{i=1}^m a_i \geq \frac{1}{2}\sum_{i=1}^n a_i$ and $\sum_{i=m}^n a_i\geq \frac{1}{2}\sum_{i=1}^n a_i$. If no such index $m$ exists, then there is an index $m'$ such that $\sum_{i=1}^{m'} a_i = \sum_{i=m'+1}^n a_i = \frac{1}{2} \sum_{i=1}^n a_i$, and in this case, the median is defined as $m'+\frac{1}{2}$. We conjecture that the median of the sequence $(k_{c,b})_{b=2}^n$ where $n=\left\lceil\frac{c+1}{2}\right\rceil$ is the same as its mode. We have confirmed this conjecture for all crossing numbers $c$ with $c\leq 10,000$. 
\begin{conjecture}
    \label{conj:median}
    Let $c\geq 3$, and let $n=\left\lceil\frac{c+1}{2}\right\rceil$. The median of the sequence $(k_{c,b})_{b=2}^n=(k_{c,2},k_{c,3},\dots,k_{c,n})$ of the number of 2-bridge knots with crossing number $c$ and braid index $b$ is $b=\left\lceil \frac{c}{3}\right\rceil +1.$
\end{conjecture}

\section{The average braid index of 2-bridge knots}
\label{sec:average}
In this section, we compute the average braid index of 2-bridge knots with crossing number $c$. Define the \textit{total braid index $\tbi(c)$} by 
\[\tbi(c) = \sum_{\mathbf{a}\in E(c)} \braid_c(K(\mathbf{a})) = \sum_{b=2}^{\left\lceil\frac{c+1}{2}\right\rceil} b\; e(c,b),\]
that is, $\tbi(c)$ is the sum of the braid indices of the knots coming from $E(c)$. Similarly, define the \textit{total palindromic braid index $\tbi_p(c)$} by 
\[\tbi_p(c) = \sum_{\mathbf{a}\in E_p(c)} \braid_c(K(\mathbf{a})) = \sum_{b=2}^{\left\lceil\frac{c+1}{2}\right\rceil} b\; e_p(c,b),\]
that is, $\tbi_p(c)$ is the sum of the braid indices of the knots coming from $E_p(c)$. Theorem \ref{thm:even} implies that
\[4\sum_{K\in\mathcal{K}_c}\braid_c(K) = \tbi(c) + \tbi_p(c).\]
Therefore the average braid index of 2-bridge knots of crossing number $c$ is
\[E(\braid_c) = \frac{\tbi(c)+\tbi_p(c)}{e(c)+e_p(c)}.\]

The following two propositions give recursive and closed formulas for $\tbi(c)$ and $\tbi_p(c)$. The closed formulas can be verified from the recursive formula by straightforward but tedious induction arguments (similar to the proofs of Corollaries \ref{cor:ec} and \ref{cor:ep}) that we leave to the reader.

\begin{proposition}
    \label{prop:tbi}
    If $c\geq 6$, then
    \begin{equation}
        \label{eq:tbirecur}
    \tbi(c) = 3\tbi(c-2) + 2\tbi(c-3) + 2^{c-3}.
    \end{equation}
    If $c\geq 3$, then
    \[\tbi(c) = \frac{(6c+22)2^{c-2}+(6c-46)(-1)^c}{27}.\]   
\end{proposition}
\begin{proof}
    The total braid index $\tbi(c)$ can be expressed as
    \begin{align*}
    \tbi(c) = & \;\sum_{b=2}^{\left\lceil\frac{c+1}{2}\right\rceil}b \; e(c,b)\\
    = & \; \sum_{b=2}^{\left\lceil\frac{c+1}{2}\right\rceil} b \;e(c-2,b) + 2\sum_{b=2}^{\left\lceil\frac{c+1}{2}\right\rceil} b\; e(c-2,b-1) + 2\sum_{b=2}^{\left\lceil\frac{c+1}{2}\right\rceil} b\; e(c-3,b-1)\\
     = & \;\sum_{b=2}^{\left\lceil\frac{c+1}{2}\right\rceil} b \;e(c-2,b) + 2\left(\sum_{b=2}^{\left\lceil\frac{c+1}{2}\right\rceil} (b-1)\; e(c-2,b-1) + \sum_{b=2}^{\left\lceil\frac{c+1}{2}\right\rceil} e(c-2,b-1)\right)\\
     & \; +2\left(\sum_{b=2}^{\left\lceil\frac{c+1}{2}\right\rceil} (b-1)\; e(c-3,b-1) + \sum_{b=2}^{\left\lceil\frac{c+1}{2}\right\rceil} e(c-3,b-1)\right)\\
     = & \; \tbi(c-2) + 2\tbi(c-2) + 2e(c-2)+2\tbi(c-3)+2e(c-3)\\
     = & \; 3\tbi(c-2)+2\tbi(c-3)+2e(c-2) + 2e(c-3).
    \end{align*}
    Corollary \ref{cor:ec} implies that
    \begin{align*}
    2e(c-2)+ 2e(c-3) = & \; \frac{4\left(2^{c-4}-(-1)^{c-4}\right)}{3} + \frac{4\left(2^{c-5}-(-1)^{c-5}\right)}{3}\\
    = & \; \frac{2^{c-2} + 2^{c-3}}{3}\\
    = & \; 2^{c-3}.
    \end{align*}
    Therefore
    \[\tbi(c) = 3\tbi(c-2) + 2\tbi(c-3) + 2^{c-3},\]
    verifying Equation \ref{eq:tbirecur}. The closed formula for $\tbi(c)$ can be verified using Equation \ref{eq:tbirecur} and induction.
\end{proof}

We now give recursive and closed formulas for $\tbi_p(c)$.
\begin{proposition}
    \label{prop:tbip}
    If $c\geq 7$, then 
\begin{equation}
\label{eq:tbiprecur}
   \tbi_p(c) = \tbi_p(c-2)+2\tbi_p(c-4)+4e_p(c-4). 
\end{equation}
If $c\geq 3$, then
\[\tbi_p(c) = \begin{cases}
    \displaystyle \frac{(3c+13)2^{\frac{c}{2}}+(12c+14)(-1)^{\frac{c}{2}}}{27} & \text{if $c$ is even,}\\
   \displaystyle \frac{(6c+14)2^{\frac{c-1}{2}}-(12c+8)(-1)^{\frac{c-1}{2}}}{27}& \text{if $c$ is odd.}
    \end{cases}\]
\end{proposition}
\begin{proof}
    The total palindromic braid index can be expressed as 
    \begin{align*}
    \tbi_p(c) = & \; \sum_{b=2}^{\left\lceil\frac{c+1}{2}\right\rceil} b\; e_p(c,b)\\
    = & \; \sum_{b=2}^{\left\lceil\frac{c+1}{2}\right\rceil} b\; e_p(c-2,b) + 2\left(\sum_{b=2}^{\left\lceil\frac{c+1}{2}\right\rceil} b\; e_p(c-4,b-2)\right)\\
    = & \; \sum_{b=2}^{\left\lceil\frac{c+1}{2}\right\rceil} b\; e_p(c-2,b) + 2\left(\sum_{b=2}^{\left\lceil\frac{c+1}{2}\right\rceil} (b-2)\; e_p(c-4,b-2) + 2 \sum_{b=2}^{\left\lceil\frac{c+1}{2}\right\rceil} e_p(c-4,b-2)\right)\\
    = & \; \tbi_p(c-2)+2\tbi_p(c-4) + 4e_p(c-4),
\end{align*}
verifying Equation \ref{eq:tbiprecur}. The closed formula for $\tbi_p(c)$ can be verified using Corollary \ref{cor:ep}, Equation \eqref{eq:tbiprecur}, and induction.
\end{proof}

Propositions \ref{prop:tbi} and \ref{prop:tbip} lead to our computation of the average braid index $E(\braid_c)$ of 2-bridge knots with crossing number $c$,  proving Theorem \ref{thm:average}.
\begin{proof}[Proof of Theorem \ref{thm:average}]
Theorem \ref{thm:even} implies that each $2$-bridge knot with crossing number $c$ is represented by either two or four elements in $E(c)$ depending on whether the corresponding elements are (anti-)palindromic or not, respectively. Thus it follows that the average braid index is
\[E(\braid_c) = \frac{\tbi(c) + \tbi_p(c)}{e(c)+e_p(c)}.\]
The result follows from Corollaries \ref{cor:ec} and \ref{cor:ep} and Propositions \ref{prop:tbi} and \ref{prop:tbip}.
\end{proof}

\section{The variance of the braid indices of 2-bridge knots}
\label{sec:variance}

In this section, we compute the variance of the braid indices of $2$-bridge knots with crossing number $c$. The variance can be expressed as $\var(\braid_c) = E(\braid_c^2) - E(\braid_c)^2$, and hence we make the following definitions. Define the \textit{total square braid index} $\tbi^2(c)$ to be the sum of the squares of the braid indices of $K(\mathbf{a})$ where $\mathbf{a}$ ranges over $E(c)$, that is,
\[\tbi^2(c) = \sum_{\mathbf{a}\in E(c)} \braid_c(K(\mathbf{a}))^2 = \sum_{b=2}^{\left\lceil\frac{c+1}{2}\right\rceil} b^2 \; e(c,b).\]
Similarly, define the \textit{total palindromic square braid index} $\tbi_p^2(c)$ to be the sum of the squares of the braid indices of $K(\mathbf{a})$ where $\mathbf{a}$ ranges over $E_p(c)$, that is,
\[\tbi_p^2(c) = \sum_{\mathbf{a}\in E_p(c)}\braid_c(K(\mathbf{a}))^2 =\sum_{b=2}^{\left\lceil\frac{c+1}{2}\right\rceil} b^2 \; e_p(c,b). \]
Theorem \ref{thm:even} implies that $E(\braid_c^2) = \frac{\tbi^2(c)+\tbi_p^2(c)}{e(c)+e_p(c)}$. 

The following two propositions give recursive and closed formulas for $\tbi^2(c)$ and $\tbi_p^2(c)$. As in Section \ref{sec:average}, the closed formulas can be verified from the recursive formula by straightforward but tedious induction arguments that we leave to the reader.
\begin{proposition}
    \label{prop:tbi2}
    If $c\geq 6$, then
\begin{align}
\label{eq:tbi2}
\begin{split}
    \tbi^2(c) = & \; 3\tbi^2(c-2) + 2\tbi^2(c-3) +4\tbi(c-2) + 4\tbi(c-3) + 2e(c-2) +2e(c-3)\\
    = & \; 3\tbi^2(c-2) + 2\tbi^2(c-3) + \frac{(6c+17)2^{c-3} + 8(-1)^c}{9}.
\end{split}
\end{align}
If $c\geq 3$, then
\[\tbi^2(c) = \frac{(3c^2+24c+37)2^{c-1}+(12c^2+30c-302)(-1)^c}{81}.\]
\end{proposition}
\begin{proof}
     If $c\geq 6$, then the total square braid index can be expressed as
 \begin{align*}
     \tbi^2(c) = & \; \sum_{b=2}^{\left\lceil\frac{c+1}{2}\right\rceil} b^2 \; e(c,b)\\
     = & \; \sum_{b=2}^{\left\lceil\frac{c+1}{2}\right\rceil} b^2(e(c-2,b) + 2e(c-2,b-1) + 2e(c-3,b-1))\\
     = & \; \sum_{b=2}^{\left\lceil\frac{c+1}{2}\right\rceil} b^2\; e(c-2,b) + 2\sum_{b=2}^{\left\lceil\frac{c+1}{2}\right\rceil} b^2\; e(c-2,b-1) + 2\sum_{b=2}^{\left\lceil\frac{c+1}{2}\right\rceil} b^2 \; e(c-3,b-1).\\
 \end{align*}
 Because $b^2 = (b-1)^2 + 2(b-1)+1$, it follows that
 \begin{align*}
     \sum_{b=2}^{\left\lceil\frac{c+1}{2}\right\rceil} b^2\; e(c-2,b-1) = & \; \sum_{b=2}^{\left\lceil\frac{c+1}{2}\right\rceil} (b-1)^2 e(c-2,b-1) + 2\sum_{b=2}^{\left\lceil\frac{c+1}{2}\right\rceil} (b-1) e(c-2,b-1) + \sum_{b=2}^{\left\lceil\frac{c+1}{2}\right\rceil} e(c-2,b-1)\\
     = & \; \tbi^2(c-2) + 2\tbi(c-2) + e(c-2).
    \end{align*}
Similarly,
\[\sum_{b=2}^{\left\lceil\frac{c+1}{2}\right\rceil} b^2 \; e(c-3,b-1) = \tbi^2(c-3) + 2\tbi(c-3) + e(c-3).\]
Therefore
\[ \tbi^2(c) = 3\tbi^2(c-2) + 2\tbi^2(c-3) +4\tbi(c-2) + 4\tbi(c-3) + 2e(c-2) +2e(c-3).\]
Hence Corollary \ref{cor:ec} and Proposition \ref{prop:tbi} imply that
\[\tbi^2(c) = 3\tbi^2(c-2) + 2\tbi^2(c-3) + \frac{(6c+17)2^{c-3} + 8(-1)^c}{9},\]
proving Equation \eqref{eq:tbi2}. The closed formula for $\tbi^2(c)$ can be verified using Equation \eqref{eq:tbi2} and induction. 
\end{proof}
The next proposition gives recursive and closed formulas for $\tbi_p^2(c)$.
\begin{proposition}
\label{prop:tbip2}
     If $c\geq 7$, then
    \begin{equation}
    \label{eq:tbip2}    
    \tbi_p^2(c) = \tbi_p^2(c-2) + 2\tbi_p^2(c-4)+8\tbi_p(c-4)+8e_p(c-4).
    \end{equation}
    If $c\geq 3$, then
    \[\tbi_p^2(c) = \begin{cases}
  \displaystyle  \frac{(6c^2+36c+14)2^{\frac{c-1}{2}}-(24c^2+24c+8)(-1)^{\frac{c-1}{2}}}{81}&\text{if $c$ is odd,}\\
   \displaystyle \frac{(3c^2+30c+43)2^{\frac{c}{2}}+(24c^2+48c+38)(-1)^{\frac{c}{2}}}{81}& \text{if $c$ is even.}
\end{cases}\]
\end{proposition}
\begin{proof}
    If $c\geq 7$, then the total palindromic square braid index can be expressed as
    \begin{align*}
        \tbi_p^2(c) = &\; \sum_{b=2}^{\left\lceil\frac{c+1}{2}\right\rceil} b^2 \; e_p(c,b)\\
        = &\; \sum_{b=2}^{\left\lceil\frac{c+1}{2}\right\rceil} b^2 e_p(c-2,b) + 2\sum_{b=2}^{\left\lceil\frac{c+1}{2}\right\rceil} b^2 e_p(c-4,b-2)\\
        = & \; \tbi_p^2(c-2) + 2\sum_{b=2}^{\left\lceil\frac{c+1}{2}\right\rceil} ((b-2)^2+4(b-2)+4) e_p(c-4,b-2)\\
        = & \; \tbi_p^2(c-2) +2\sum_{b=2}^{\left\lceil\frac{c+1}{2}\right\rceil} (b-2)^2 e_p(c-4,b-2)\\
        & \; + 8\sum_{b=2}^{\left\lceil\frac{c+1}{2}\right\rceil} (b-2) e_p(c-4,b-2) + 8 \sum_{b=2}^{\left\lceil\frac{c+1}{2}\right\rceil} e_p(c-4,b-2)\\
       = & \; \tbi_p^2(c-2) + 2\tbi_p^2(c-4)+8\tbi_p(c-4)+8e_p(c-4),
    \end{align*}
    proving Equation \eqref{eq:tbip2}. The closed formula for $\tbi_p^2(c)$ can be verified using Corollary \ref{cor:ep}, Proposition \ref{prop:tbip}, and induction.
\end{proof}
The closed formulas for $\tbi^2(c)$ and $\tbi_p^2(c)$ in Propositions \ref{prop:tbi2} and \ref{prop:tbip2} lead to our computation of $\var(\braid_c)$, proving Theorem \ref{thm:variance} and concluding the paper.
\begin{proof}[Proof of Theorem \ref{thm:variance}]
The variance $\var(\braid_c)$ of the braid indices of 2-bridge knots with crossing number $c$ satisfies
\[\var(\braid_c) = E(\braid_c^2) - E(\braid_c)^2 = \frac{\tbi^2(c)+\tbi_p^2(c)}{e(c)+e_p(c)} +\left(\frac{\tbi(c)+\tbi_p(c)}{e(c)+e_p(c)}\right)^2.\]
Theorem \ref{thm:average}, Corollaries \ref{cor:ec} and \ref{cor:ep}, and Propositions \ref{prop:tbi2} and \ref{prop:tbip2} imply the result. 
\end{proof}

\bibliographystyle{amsalpha}
\bibliography{Braid}

\end{document}